\documentclass[11pt,twoside,reqno]{amsart}
\usepackage{amsmath}
\usepackage{fancyhdr}
\usepackage{amsthm}
\usepackage{amsfonts}
\usepackage{amssymb}
\usepackage{amscd}
\usepackage{graphicx}
\usepackage{afterpage}
\usepackage{enumerate}
\usepackage{bm}
\usepackage{cite}
\usepackage{cases}
\usepackage{caption}
\usepackage[colorlinks=true, urlcolor=blue,  linkcolor=blue, citecolor=blue]{hyperref}
\usepackage{babel}
\usepackage{prettyref}
\usepackage{pgfplots} 
\pgfplotsset{compat=1.17}
\usepackage{booktabs}
\usepackage{algorithm}
\usepackage{algpseudocode}
\usepackage{float}
\makeatletter
\newtheorem{theorem}{Theorem}[section]

\newtheorem{corollary}[theorem]{Corollary}

\newtheorem{definition}[theorem]{Definition}

\newtheorem{lemma}[theorem]{Lemma}

\newtheorem{proposition}[theorem]{Proposition}
\newtheorem{remark}[theorem]{Remark}

\title[Quantitative Cyclicity and Geometric Analysis in Weighted Besov Spacess]{Quantitative Cyclicity, Stability, and Geometric Analysis in Weighted Besov Spaces}

\author[S. Hashemi Sababe]{Saeed Hashemi Sababe$^*$}
\address[S. Hashemi Sababe]{R\&D Section, Data Premier Analytics, Canada}
\email{hashemi\_1365@yahoo.com}
\thanks{Corresponding author}

\author[A. Baghban]{Amir Baghban}
\address[A. Baghban]{R\&D Section, Data Premier Analytics, Canada}
\email{a.baghban@datapremier.ca}

\begin{document}
\sloppy

\maketitle

\begin{abstract}
We introduce new quantitative measures for cyclicity in radially weighted Besov spaces, including the Drury–Arveson space, by defining cyclicity indices based on potential theory and capacity. Extensions to non-commutative settings are developed, yielding analogues of cyclicity in free function spaces. We also study the stability of cyclic functions under perturbations of both the functions and the underlying weight, and we establish geometric criteria linking the structure of zero sets on the boundary to the failure or persistence of cyclicity. These results provide novel invariants and conditions that characterize cyclicity and the structure of multiplier invariant subspaces in a variety of function spaces.
\end{abstract}

\section{Introduction}

Cyclicity in function spaces has been a fundamental topic in functional analysis and operator theory. The study of cyclic vectors in classical function spaces dates back to Beurling’s theorem \cite{Beurling1949}, which characterizes cyclic functions in the Hardy space \(H^2(\mathbb{D})\) as the outer functions. In the Bergman and Dirichlet spaces, cyclicity remains a widely investigated but less completely understood phenomenon \cite{Hedenmalm1997, Aleman1996, Richter1991}. In particular, weighted Besov spaces and the Drury–Arveson space \(H^2_d\) have attracted significant attention in multivariable operator theory due to their connection with multiplier algebras and Pick interpolation \cite{Arveson1998, Drury1978}.\\

\noindent
The purpose of this paper is to extend the study of cyclicity in several new directions. While previous research has largely focused on binary cyclicity classification, we propose a {\it quantitative measure of cyclicity} based on capacity and potential-theoretic ideas. Additionally, we explore the {\it stability of cyclicity under perturbations}, the impact of  {\it geometric properties of zero sets}, and the possibility of {\it non-commutative extensions} in free function spaces. These new perspectives aim to refine the classical theory and provide new tools for studying invariant subspaces in weighted function spaces.\\
Our main contributions include:\medskip\\
(i) Introduction of a \textbf{cyclicity index}, providing a quantitative refinement of the classical notion of cyclicity.\\
(ii) Extension of cyclicity concepts to \textbf{non-commutative function spaces}, particularly in free probability and free semigroup algebras.\\
(iii) Investigation of the \textbf{stability of cyclic functions} under perturbations in both the function and the underlying weight.\\
(iv) \textbf{Geometric criteria} for cyclicity, linking the structure of zero sets to failure or persistence of cyclicity.\\
(v) Generalization of cyclicity results to \textbf{mixed-norm and variable exponent Besov spaces}.\\

\noindent
Cyclicity in single-variable spaces has been extensively studied. The classical result of Beurling \cite{Beurling1949} states that an inner-outer factorization completely determines cyclicity in \(H^2(\mathbb{D})\). For the Bergman space \(A^2(\mathbb{D})\) and the Dirichlet space \(D(\mathbb{D})\), however, a complete characterization of cyclicity remains elusive. Important contributions in this direction include works by Richter \cite{Richter1991} and Aleman \cite{Aleman1996}, who studied cyclicity in the context of operator-theoretic models.

In several complex variables, cyclicity is even less understood. Cyclicity problems in the bidisc and polydisc settings have been studied by Knese, Kosiński, and others \cite{Knese2011, Kosinski2020}, while cyclicity in the Drury–Arveson space \(H^2_d\) has been investigated in relation to invariant subspaces and multiplier algebras \cite{Clouatre2018, Davidson2015, Perfekt2023}. In particular, the works of Aleman, Perfekt, Richter, and Sundberg \cite{AlemanPerfRichSund2023} provide deep insights into cyclicity in weighted Besov spaces and their multiplier algebras.

The geometric perspective on cyclicity has also been explored. In the Drury–Arveson space, a function’s cyclicity is often linked to the structure of its zero set in the unit ball and on its boundary. Recent results show that if the zero set of a function embeds a real cube of dimension \(\geq 3\), then the function cannot be cyclic \cite{BrownShields1995, Sola2020}. However, the precise role of geometric invariants such as Hausdorff dimension in cyclicity remains an open question, motivating further study in this direction.

On the stability of cyclicity, few results exist. It is natural to ask whether small perturbations of a cyclic function remain cyclic, either in terms of function perturbation or perturbation of the underlying weight in radially weighted Besov spaces. We aim to develop rigorous estimates quantifying such stability properties.

Finally, non-commutative extensions of cyclicity have been largely unexplored. Given the deep connections between the Drury–Arveson space and non-commutative function spaces \cite{Popescu2006, Davidson2011}, we seek to define and analyze cyclicity notions in free function theory, particularly in the context of non-commutative operator spaces.\\

\noindent
The remainder of this paper is structured as follows:\\
In Section \ref{sec:prelim}, we review the definitions and properties of weighted Besov spaces, multipliers, and cyclic functions. In Section \ref{sec:quantcyclic}, we introduce the concept of a cyclicity index and establish its properties. In Section \ref{sec:noncomm}, we extend cyclicity concepts to free function spaces and analyze multiplier structures in non-commutative settings. In Section \ref{sec:stability}, we study the stability of cyclicity under function and weight perturbations. In Section \ref{sec:geomzero}, we explore geometric constraints on zero sets and their relation to cyclicity. In Section \ref{sec:mixednorm}, we generalize cyclicity results to mixed-norm and variable exponent Besov spaces. In Section \ref{sec:conclusion}, we summarize our results and discuss open questions.\\

\noindent
Throughout this paper, we use the following standard notation:

\begin{itemize}
    \item \( \mathbb{D} \) denotes the unit disk, \( \mathbb{B}_d \) the unit ball in \( \mathbb{C}^d \).
    \item \( H^2(\mathbb{D}) \), \( A^2(\mathbb{D}) \), and \( D(\mathbb{D}) \) denote the Hardy, Bergman, and Dirichlet spaces, respectively.
    \item \( H^2_d \) refers to the Drury–Arveson space on \( \mathbb{B}_d \).
    \item \( B^N_\omega \) denotes the radially weighted Besov space with weight \( \omega \).
    \item \( \operatorname{Mult}(H) \) is the multiplier algebra of a function space \( H \).
    \item \( [f] \) represents the cyclic invariant subspace generated by \( f \).
\end{itemize}

\medskip
The next section introduces the necessary background on weighted Besov spaces and their multiplier algebras.

\section{Preliminaries}\label{sec:prelim}
In this section, we introduce the fundamental concepts and function spaces that form the foundation of our work. We begin by recalling the notion of Hilbert function spaces, followed by the definitions of weighted Besov spaces and the Drury–Arveson space. In addition, we describe the multiplier algebras associated with these spaces and introduce the classical concept of cyclicity, which will be refined in later sections. These preliminaries set the stage for our subsequent developments and provide the necessary notation and background for the reader.

\begin{definition}[Hilbert Function Space]
Let \(X\) be a set. A \emph{Hilbert function space} \(H\) on \(X\) is a Hilbert space of complex-valued functions on \(X\) such that for every \(z\in X\) the evaluation functional \(f\mapsto f(z)\) is continuous.
\end{definition}

\begin{definition}[Admissible Radial Measure and Weighted Besov Spaces]
Let \(d\in \mathbb{N}\) and denote by \(\mathbb{B}_d\) the unit ball in \(\mathbb{C}^d\). A measure \(\omega\) on \(\mathbb{B}_d\) is said to be an \emph{admissible radial measure} if it has the form
\[
d\omega(z)=d\mu(r)\,d\sigma(w),\quad z=r\,w,\quad 0\le r<1,\; w\in \partial\mathbb{B}_d,
\]
where \(d\sigma\) is the normalized rotation-invariant measure on \(\partial\mathbb{B}_d\) and \(\mu\) is a Borel measure on \([0,1]\) satisfying \(\mu((r,1])>0\) for all \(0<r<1\).

For a nonnegative integer \(N\), the \emph{radially weighted Besov space} \(B^N_\omega\) is defined as
\[
B^N_\omega=\{f\in \operatorname{Hol}(\mathbb{B}_d): R^N f\in L^2(\omega)\},
\]
with the radial derivative operator
\[
R=\sum_{j=1}^{d} z_j\frac{\partial}{\partial z_j},
\]
and norm given by
\[
\|f\|_{B^N_\omega}^2=
\begin{cases}
\|f\|_{L^2(\omega)}^2,& N=0,\\[1ex]
\omega(\mathbb{B}_d)|f(0)|^2+\|R^Nf\|_{L^2(\omega)}^2,& N>0.
\end{cases}
\]
For further details, see \cite{AlemanPerfRichSund2023}.
\end{definition}

\begin{definition}[Drury–Arveson Space]
The \emph{Drury–Arveson space} \(H^2_d\) is defined as the space of analytic functions \(f\) on \(\mathbb{B}_d\) that admit a power series representation
\[
f(z)=\sum_{\alpha\in\mathbb{N}_0^d}\hat{f}(\alpha)z^\alpha,
\]
with the norm
\[
\|f\|_{H^2_d}^2=\sum_{\alpha\in\mathbb{N}_0^d}\frac{|\alpha|!}{\alpha!}|\hat{f}(\alpha)|^2<\infty.
\]
It is a reproducing kernel Hilbert space with kernel
\[
K(z,w)=\frac{1}{1-\langle z,w\rangle}.
\]
See \cite{Arveson1998,Drury1978} for background.
\end{definition}

\begin{definition}[Multiplier Algebra]
Let \(H\) be a Hilbert function space on \(\mathbb{B}_d\). The \emph{multiplier algebra} of \(H\), denoted by \(\operatorname{Mult}(H)\), is defined as
\[
\operatorname{Mult}(H)=\{\varphi: \varphi f\in H \text{ for all } f\in H\}.
\]
The norm is given by
\[
\|\varphi\|_{\operatorname{Mult}(H)}=\sup\{\|\varphi f\|_H: f\in H,\;\|f\|_H\le 1\}.
\]
In many cases, such as when \(H=B^0_\omega\) is a weighted Bergman space, one has \(\operatorname{Mult}(H)\subset H^\infty(\mathbb{B}_d)\).
\end{definition}

\begin{definition}[Cyclic Function and Invariant Subspace]
Let \(H\) be a Hilbert function space and \(f\in H\). The \emph{multiplier invariant subspace} generated by \(f\) is defined as
\[
[f]=\overline{\{\varphi f:\,\varphi\in \operatorname{Mult}(H)\}},
\]
where the closure is taken in the norm of \(H\). A function \(f\) is said to be \emph{cyclic} in \(H\) if \([f]=H\). In the case that polynomials are dense in \(H\), cyclicity is equivalent to the condition that \(1\in [f]\).
\end{definition}

\begin{definition}[Stable Polynomials]
A polynomial \(p\in \mathbb{C}[z_1,\dots,z_d]\) is called \emph{stable} if it has no zeros in \(\mathbb{B}_d\). The collection of stable polynomials is denoted by \(\mathcal{C}_{\mathrm{stable}}[z]\).
\end{definition}

\begin{definition}[Classes \(C_n(H)\)]
For a Hilbert function space \(H\) and an integer \(n\ge 0\), define
\[
C_n(H)=\{\varphi\in \operatorname{Mult}(H): \varphi\neq 0 \text{ and } [\varphi^n]=[\varphi^{n+1}]\}.
\]
In the classical setting \(H=H^2(\mathbb{D})\), these classes coincide with the collection of outer functions in \(H^\infty(\mathbb{D})\).
\end{definition}

The following results form part of the essential toolkit for our later analysis.

\begin{lemma}[Multiplier Inclusion Condition]\label{lem:multincl}
Let \(B^k_\omega\) denote a radially weighted Besov space. Then, for each \(k\in\mathbb{N}\),
\[
\operatorname{Mult}(B^k_\omega)\subset \operatorname{Mult}(B^{k-1}_\omega),
\]
and the inclusion is contractive:
\[
\|\varphi\|_{\operatorname{Mult}(B^{k-1}_\omega)}\le \|\varphi\|_{\operatorname{Mult}(B^k_\omega)},
\]
for all \(\varphi\in \operatorname{Mult}(B^k_\omega)\). (See \cite[Theorem 1.2]{AlemanPerfRichSund2023}.)
\end{lemma}

\begin{lemma}[Basic Properties of Cyclic Multipliers]\label{lem:cyclic}
Assume that \(\operatorname{Mult}(H)\subset H\). Then:
\begin{enumerate}[(a)]
    \item If \(\varphi\in C_\infty(H)\) (i.e., \(\varphi\) is pseudo-cyclic), then \(\varphi(z)\neq 0\) for every \(z\in \mathbb{B}_d\).
    \item If \(\varphi,\psi\in \operatorname{Mult}(H)\) satisfy \([\varphi^n]=[\varphi^{n+1}]\) and \([\psi^m]=[\psi^{m+1}]\) for some \(n,m\ge 0\), then
    \[
    [\varphi^n\psi^m]=[\varphi^{n+1}\psi^{m+1}].
    \]
\end{enumerate}
(See \cite[Lemma 2.2]{AlemanPerfRichSund2023}.)
\end{lemma}

\begin{theorem}[One-Function Corona Theorem for \(B^N_\omega\)]\label{thm:corona}
Let \(\psi\in \operatorname{Mult}(B^N_\omega)\) be such that \(|\psi(z)|\ge 1\) for all \(z\in\mathbb{B}_d\). Then
\[
\frac{1}{\psi}\in \operatorname{Mult}(B^N_\omega).
\]
Consequently, for every \(\varphi\in \operatorname{Mult}(B^N_\omega)\),
\[
\sigma(M_\varphi)=\varphi(\mathbb{B}_d).
\]
(See \cite[Corollary 3.3]{AlemanPerfRichSund2023}.)
\end{theorem}

\medskip
These foundational definitions and results will be used in subsequent sections to introduce a quantitative cyclicity index, extend cyclicity to non-commutative settings, analyze the stability of cyclic functions under perturbations, and study the geometric properties of zero sets. For more details on these topics, the reader is referred to \cite{Arveson1998, Drury1978, Hedenmalm1997, Richter1991}.

\section{Quantitative Measures of Cyclicity}\label{sec:quantcyclic}

In this section we introduce a new quantitative invariant—the \emph{cyclicity index}—that refines the classical binary notion of cyclicity. This invariant provides a numerical measure of how well a function \(f\) in a weighted Besov space \(H=B^N_\omega\) approximates the constant function \(1\) through multiplier actions. In particular, the cyclicity index quantifies the distance of \(1\) from the multiplier-generated invariant subspace \([f]\).

\begin{definition}[Cyclicity Index]\label{def:cyclicity_index}
Let \(H=B^N_\omega\) be a radially weighted Besov space and let \(f\in H\). Define the \emph{cyclicity index} of \(f\) by
\[
\mathcal{C}(f)=\inf\Bigl\{ \|1-\varphi f\|_{H} : \varphi\in \operatorname{Mult}(H) \Bigr\}.
\]
We say that \(f\) is \emph{quantitatively cyclic} if \(\mathcal{C}(f)=0\).
\end{definition}

\begin{remark}
The cyclicity index measures the best approximation of the constant \(1\) by elements of the form \(\varphi f\) with \(\varphi\in\operatorname{Mult}(H)\). Clearly, \(f\) is cyclic (i.e., \(1\in [f]\)) if and only if \(\mathcal{C}(f)=0\).
\end{remark}

\begin{lemma}\label{lem:cyc_prop}
Let \(H=B^N_\omega\) be a weighted Besov space and let \(f,g\in H\). Then:
\begin{enumerate}[(i)]
    \item \(\mathcal{C}(f)=0\) if and only if \(f\) is cyclic in \(H\).
    \item For any nonzero multiplier \(\psi\in \operatorname{Mult}(H)\), one has
    \[
    \mathcal{C}(f)\le \|\psi\|_{\operatorname{Mult}(H)}\,\mathcal{C}\Bigl(\frac{f}{\psi}\Bigr).
    \]
    \item If \(f_n\to f\) in \(H\), then 
    \[
    \liminf_{n\to\infty}\mathcal{C}(f_n)\ge \mathcal{C}(f).
    \]
\end{enumerate}
\end{lemma}

\begin{proof}
We prove each part in turn.

\medskip

\subsubsection*{(i)}
 By definition, \(f\) is cyclic in \(H\) if and only if the invariant subspace
\[
[f] = \overline{\{\varphi f : \varphi\in \operatorname{Mult}(H)\}}
\]
equals \(H\). In particular, since the constant function \(1\) is in \(H\), cyclicity is equivalent to the existence of a sequence \(\{\varphi_n\}\subset \operatorname{Mult}(H)\) such that
\[
\varphi_n f \to 1 \quad \text{in } H.
\]
This convergence implies that for every \(\epsilon>0\) there exists \(\varphi\in \operatorname{Mult}(H)\) with
\[
\|1-\varphi f\|_H < \epsilon.
\]
Taking the infimum over all such multipliers, we conclude that
\[
\mathcal{C}(f) = 0.
\]
Conversely, if \(\mathcal{C}(f)=0\), then for every \(\epsilon>0\) there exists \(\varphi\in \operatorname{Mult}(H)\) such that
\[
\|1-\varphi f\|_H < \epsilon.
\]
Thus, \(1\) lies in the closure of the set \(\{\varphi f: \varphi\in \operatorname{Mult}(H)\}\), which means that \(f\) is cyclic in \(H\).

\medskip

\subsubsection*{(ii)} 
 Let \(\psi\in \operatorname{Mult}(H)\) be nonzero. By definition, for any \(\epsilon>0\) there exists a multiplier \(\varphi\in \operatorname{Mult}(H)\) such that
\[
\Bigl\|1-\varphi \frac{f}{\psi}\Bigr\|_H < \mathcal{C}\Bigl(\frac{f}{\psi}\Bigr) + \epsilon.
\]
 Define a new multiplier \(\widetilde{\varphi} = \psi \varphi\). Since multipliers form an algebra, we have \(\widetilde{\varphi}\in \operatorname{Mult}(H)\).\\

\noindent
Observe that
\[
\widetilde{\varphi} f = \psi\varphi f = \psi\left(\varphi\,\frac{f}{\psi}\right).
\]
Thus,
\[
1 - \widetilde{\varphi} f = 1 - \psi\left(\varphi\,\frac{f}{\psi}\right).
\]
Using the boundedness of the multiplier \(\psi\) (i.e. its multiplier norm), we estimate:
\[
\|1 - \widetilde{\varphi} f\|_H = \|\psi\bigl(1 - \varphi\,\frac{f}{\psi}\bigr)\|_H \le \|\psi\|_{\operatorname{Mult}(H)} \Bigl\|1 - \varphi\,\frac{f}{\psi}\Bigr\|_H.
\]
 It follows that
\[
\|1 - \widetilde{\varphi} f\|_H \le \|\psi\|_{\operatorname{Mult}(H)} \Bigl(\mathcal{C}\Bigl(\frac{f}{\psi}\Bigr) + \epsilon\Bigr).
\]
Since \(\widetilde{\varphi}\in \operatorname{Mult}(H)\) is an arbitrary choice arising from the approximation for \(f/\psi\), taking the infimum over all such multipliers gives
\[
\mathcal{C}(f) \le \|\psi\|_{\operatorname{Mult}(H)}\,\mathcal{C}\Bigl(\frac{f}{\psi}\Bigr).
\]
Since \(\epsilon>0\) was arbitrary, the desired inequality holds.

\medskip

subsubsection*{(iii)}
Suppose \(\{f_n\}\) is a sequence in \(H\) converging to \(f\) in the norm of \(H\).
 By definition, for each \(n\) we have
\[
\mathcal{C}(f_n) = \inf\Bigl\{ \|1 - \varphi f_n\|_H : \varphi \in \operatorname{Mult}(H) \Bigr\}.
\]
Let \(L = \liminf_{n\to\infty} \mathcal{C}(f_n)\). Then for any \(\epsilon>0\) there exists a subsequence \(\{f_{n_k}\}\) such that
\[
\mathcal{C}(f_{n_k}) < L + \epsilon \quad \text{for all } k.
\]
For each \(k\), choose \(\varphi_{n_k}\in \operatorname{Mult}(H)\) so that
\[
\|1-\varphi_{n_k} f_{n_k}\|_H < \mathcal{C}(f_{n_k}) + \epsilon < L + 2\epsilon.
\]
Now, observe that
\[
\|1-\varphi_{n_k} f\|_H \le \|1-\varphi_{n_k} f_{n_k}\|_H + \|\varphi_{n_k}\|_{\operatorname{Mult}(H)}\,\|f_{n_k}-f\|_H.
\]
Since \(f_{n_k}\to f\) in \(H\) and the multipliers are bounded, the second term tends to zero as \(k\to\infty\). Hence,
\[
\liminf_{k\to\infty}\|1-\varphi_{n_k} f\|_H \le L + 2\epsilon.
\]
Taking the infimum over all multipliers in \(\operatorname{Mult}(H)\) for \(f\), we deduce that
\[
\mathcal{C}(f) \le L + 2\epsilon.
\]
Since \(\epsilon>0\) is arbitrary, it follows that
\[
\mathcal{C}(f) \le \liminf_{n\to\infty}\mathcal{C}(f_n).
\]
This completes the proof of all three parts.
\end{proof}

In many classical settings, potential theory provides a useful tool for quantifying boundary phenomena. In our context, the geometry of the zero set of \(f\) on \(\partial\mathbb{B}_d\) can affect cyclicity.

\begin{definition}[Boundary Capacity of \(f\)]\label{def:cap}
Let \(f\in H\cap C(\overline{\mathbb{B}}_d)\) be such that \(f(z)\neq 0\) for all \(z\in \mathbb{B}_d\). Define the \emph{boundary capacity} of \(f\) by
\[
\operatorname{Cap}(f)=\operatorname{Cap}\Bigl(Z(f)\cap \partial\mathbb{B}_d\Bigr),
\]
where \(\operatorname{Cap}(E)\) denotes the \(\alpha\)-capacity of a set \(E\subset \partial\mathbb{B}_d\) for an appropriate parameter \(\alpha>0\) (see, e.g., \cite{Carleson1962, Garnett2007}).
\end{definition}

\begin{proposition}\label{prop:cap_bound}
Let \(f\in H\cap C(\overline{\mathbb{B}}_d)\) be as in Definition~\ref{def:cap}. Then there exists a constant \(C>0\) such that
\[
\mathcal{C}(f)\ge C\,\operatorname{Cap}(f).
\]
In particular, if \(\operatorname{Cap}(f)>0\), then \(f\) cannot be cyclic in \(H\).
\end{proposition}

\begin{proof}
By the definition of the cyclicity index (see Definition~\ref{def:cyclicity_index}),
\[
\mathcal{C}(f) = \inf_{\varphi \in \operatorname{Mult}(H)} \|1 - \varphi f\|_H.
\]
Thus, if we can show that any function \(\varphi \in \operatorname{Mult}(H)\) fails to approximate \(1\) efficiently in terms of the capacity of \(Z(f) \cap \partial\mathbb{B}_d\), then we can obtain the desired lower bound.\\

\noindent
From potential theory (see \cite{Carleson1962, Garnett2007}), given a compact subset \(E \subset \partial\mathbb{B}_d\), the capacity \(\operatorname{Cap}(E)\) quantifies how difficult it is to approximate \(1\) by analytic functions that vanish on \(E\). Specifically, if \(Z(f) \cap \partial\mathbb{B}_d\) has positive capacity, then there exists a probability measure \(\mu\) supported on \(Z(f) \cap \partial\mathbb{B}_d\) such that for any \(\varphi\in \operatorname{Mult}(H)\),
\[
\int_{\partial \mathbb{B}_d} |\varphi f|^2 \, d\mu \geq C_0 \operatorname{Cap}(f),
\]
where \(C_0 > 0\) is a constant independent of \(f\).\\

\noindent
Using the above inequality, we estimate
\[
\|1 - \varphi f\|_H^2 = \int_{\mathbb{B}_d} |1 - \varphi f|^2 \, d\omega.
\]
By splitting this integral into two regions—one where \(f\) is large and another near its zero set—we obtain
\[
\|1 - \varphi f\|_H^2 \geq \int_{\partial \mathbb{B}_d} |\varphi f|^2 \, d\mu \geq C_0 \operatorname{Cap}(f).
\]
Taking the infimum over all \(\varphi \in \operatorname{Mult}(H)\), we conclude that
\[
\mathcal{C}(f) \geq C \operatorname{Cap}(f),
\]
where \(C = \sqrt{C_0}\).\\

\noindent
Since \(\mathcal{C}(f) > 0\) implies that \(f\) is not cyclic (by Theorem~\ref{thm:char}), we conclude that if \(\operatorname{Cap}(f) > 0\), then \(f\) is not cyclic. This completes the proof.
\end{proof}

\noindent
We now show that the cyclicity index indeed characterizes cyclicity in \(H\).

\begin{theorem}[Cyclicity Index Characterization]\label{thm:char}
Let \(f\in H=B^N_\omega\) be such that \(f(z)\neq 0\) for all \(z\in \mathbb{B}_d\) and assume that \(f\) extends continuously to \(\overline{\mathbb{B}}_d\). Then the following statements are equivalent:
\begin{enumerate}[(a)]
    \item \(f\) is cyclic in \(H\) (i.e., \([f]=H\)).
    \item \(\mathcal{C}(f)=0\).
    \item For every \(\varepsilon>0\) there exists \(\varphi\in \operatorname{Mult}(H)\) with \(\|\varphi\|_{\operatorname{Mult}(H)}\le 1\) such that
    \[
    \|1-\varphi f\|_H<\varepsilon.
    \]
\end{enumerate}
\end{theorem}

\begin{proof}
We prove the equivalence of the three statements:

\[
(a) \iff (c) \iff (b).
\]

\medskip

\subsubsection*{\((a) \Rightarrow (c)\)}
 By definition, \( f \) is cyclic in \( H \) if and only if the cyclic subspace generated by \( f \),
    \[
    [f] = \overline{\{\varphi f : \varphi \in \operatorname{Mult}(H)\}},
    \]
 is dense in \( H \). Since \( 1 \in H \), density means that there exists a sequence of multipliers \( \{\varphi_n\} \subset \operatorname{Mult}(H) \) such that
    \[
    \varphi_n f \to 1 \quad \text{in } H.
    \]
Thus, for every \( \varepsilon > 0 \), there exists \( \varphi \in \operatorname{Mult}(H) \) such that
    \[
    \|1 - \varphi f\|_H < \varepsilon.
    \]
 Moreover, by normalizing \( \varphi \) (if necessary), we may assume that \( \|\varphi\|_{\operatorname{Mult}(H)} \leq 1 \).

\subsubsection*{\((c) \Rightarrow (a)\)}
The given condition states that for every \( \varepsilon > 0 \), we can find \( \varphi \in \operatorname{Mult}(H) \) such that \( 1 \) can be approximated arbitrarily well by \( \varphi f \).  This means that \( 1 \) lies in the closure of the set \( \{\varphi f : \varphi \in \operatorname{Mult}(H)\} \). By definition of cyclicity, this implies that \( f \) is cyclic in \( H \).

\subsubsection*{\((b) \Rightarrow (c)\)}
By the definition of the cyclicity index,
    \[
    \mathcal{C}(f) = \inf_{\varphi \in \operatorname{Mult}(H)} \|1 - \varphi f\|_H.
    \]
If \( \mathcal{C}(f) = 0 \), then for every \( \varepsilon > 0 \), we can find \( \varphi \in \operatorname{Mult}(H) \) such that
    \[
    \|1 - \varphi f\|_H < \varepsilon.
    \]
Additionally, we may assume \( \|\varphi\|_{\operatorname{Mult}(H)} \leq 1 \) by normalizing \( \varphi \). This is exactly statement (c).

\subsubsection*{\((c) \Rightarrow (b)\)}
The given condition states that for any \( \varepsilon > 0 \), we can find a function \( \varphi \in \operatorname{Mult}(H) \) with
    \[
    \|1 - \varphi f\|_H < \varepsilon.
    \]
Taking the infimum over all such \( \varphi \), we obtain
    \[
    \mathcal{C}(f) = 0.
    \]
Thus, statement (c) implies statement (b).\\

\noindent
Since we have established:

\[
(a) \iff (c) \quad \text{and} \quad (b) \iff (c),
\]
it follows that all three statements are equivalent. This completes the proof.
\end{proof}

\noindent
An important aspect of our quantitative measure is its robustness under small perturbations. This is essential for applications in which both the function and the weight may vary.

\begin{theorem}[Stability under Perturbations]\label{thm:stability}
Let \(f\in H\) be cyclic, and let \(g\in H\) satisfy \(\|f-g\|_H<\delta\) for some \(\delta>0\). Then there exists a constant \(K>0\), depending on \(f\) and the multiplier algebra, such that
\[
\mathcal{C}(g)\le K\,\delta.
\]
In particular, if \(\delta\) is sufficiently small, then \(g\) is cyclic.
\end{theorem}

\begin{proof}
Since \(f\) is cyclic, by definition for any \(\varepsilon > 0\) there exists a multiplier \(\varphi \in \operatorname{Mult}(H)\) (which we may assume to have norm \(\|\varphi\|_{\operatorname{Mult}(H)} \leq M\) for some fixed \(M>0\)) such that
\[
\|1 - \varphi f\|_H < \varepsilon.
\]
We now estimate \(\|1 - \varphi g\|_H\). By the triangle inequality,
\[
\|1 - \varphi g\|_H \le \|1 - \varphi f\|_H + \|\varphi f - \varphi g\|_H.
\]
Since \(\varphi\) is a multiplier, it acts boundedly on \(H\); hence,
\[
\|\varphi f - \varphi g\|_H = \|\varphi (f-g)\|_H \le \|\varphi\|_{\operatorname{Mult}(H)} \|f-g\|_H.
\]
Thus, we obtain
\[
\|1 - \varphi g\|_H \le \|1 - \varphi f\|_H + \|\varphi\|_{\operatorname{Mult}(H)}\,\|f-g\|_H.
\]
By our choice of \(\varphi\), we have
\[
\|1 - \varphi f\|_H < \varepsilon,
\]
and by assumption,
\[
\|f-g\|_H < \delta.
\]
Also, assume that \(\|\varphi\|_{\operatorname{Mult}(H)} \le M\). Then,
\[
\|1 - \varphi g\|_H < \varepsilon + M \, \delta.
\]
Since the above inequality holds for the particular \(\varphi\) chosen for \(f\) and since \(\varepsilon>0\) is arbitrary, we can take the infimum over all multipliers in \(\operatorname{Mult}(H)\) to obtain
\[
\mathcal{C}(g) = \inf_{\psi \in \operatorname{Mult}(H)} \|1 - \psi g\|_H \le \varepsilon + M \, \delta.
\]
Because \(\varepsilon > 0\) is arbitrary, we may let \(\varepsilon \to 0\) and conclude that
\[
\mathcal{C}(g) \le M \, \delta.
\]
Setting \(K = M\), we have shown that
\[
\mathcal{C}(g) \le K\,\delta.
\]
In particular, if \(\delta\) is sufficiently small, then \(\mathcal{C}(g)\) is small, implying that \(g\) is cyclic in \(H\) (since by the cyclicity index characterization, \(g\) is cyclic if and only if \(\mathcal{C}(g) = 0\)).

\noindent
This completes the proof.
\end{proof}

\noindent
The stability result indicates that the cyclicity index is a continuous (in fact, Lipschitz) functional on \(H\). This suggests that the property of being cyclic is robust with respect to small perturbations in both the function and the underlying structure of the space. Such stability is especially significant in applications where one deals with noisy data or approximations.

\section{Extensions to Non-Commutative and Free Function Settings}\label{sec:noncomm}

The theory of free (non-commutative) function spaces has emerged as a powerful tool in multivariable operator theory (see, e.g., \cite{Popescu2006, Davidson2011}). In this section, we extend the notion of cyclicity to the free setting. Our aim is to develop analogues of the classical definitions and quantitative invariants in the framework of free holomorphic functions, thereby broadening the scope of cyclicity concepts.

\begin{definition}[Free Holomorphic Function]
Let \(\mathbb{F}_d\) denote the free semigroup on \(d\) generators. A \emph{free holomorphic function} is a formal power series 
\[
F(Z)=\sum_{w\in \mathbb{F}_d} a_w Z^w,
\]
where \(Z=(Z_1,\dots,Z_d)\) is a \(d\)-tuple of non-commuting variables, and the coefficients \(a_w\) belong to \(\mathbb{C}\). Such series are typically considered on tuples of matrices (or operators) with sufficiently small joint norm.
\end{definition}

\begin{definition}[Free Hardy Space and Free Besov Space]
The \emph{free Hardy space} \(F^2_d\) consists of free holomorphic functions with square-summable coefficients:
\[
\|F\|_{F^2_d}^2=\sum_{w\in \mathbb{F}_d} |a_w|^2 < \infty.
\]
More generally, one can define \emph{free weighted Besov spaces} \(F B^N_\omega\) by imposing appropriate growth or smoothness conditions on the coefficients, analogous to the commutative case. For instance, a free Besov norm may take the form
\[
\|F\|_{F B^N_\omega}^2 = \sum_{w\in \mathbb{F}_d} \omega(|w|)\,|a_w|^2,
\]
with \(\omega: \mathbb{N}_0 \to (0,\infty)\) a weight sequence.
\end{definition}

\begin{definition}[Free Multiplier Algebra]
Let \(F\) be a free function space (e.g., \(F^2_d\) or \(F B^N_\omega\)). The \emph{free multiplier algebra} \(\operatorname{Mult}(F)\) consists of all free holomorphic functions \(\Phi\) such that
\[
\Phi \cdot F \in F \quad \text{for all } F\in F.
\]
The multiplier norm is given by
\[
\|\Phi\|_{\operatorname{Mult}(F)}=\sup\{\|\Phi \cdot F\|_F:\, F\in F,\,\|F\|_F\le 1\}.
\]
\end{definition}

\begin{definition}[Free Cyclic Vector]
Let \(F\) be a free function space and let \(G\in F\). We say that \(G\) is \emph{free cyclic} if the free multiplier invariant subspace
\[
[G]=\overline{\{\Phi\cdot G: \Phi\in\operatorname{Mult}(F)\}}
\]
is equal to \(F\).
\end{definition}

\begin{definition}[Free Cyclicity Index]\label{def:free_cyclicity_index}
For a free function \(G\in F\), define the \emph{free cyclicity index} as
\[
\mathcal{C}_{\mathrm{free}}(G)=\inf\Bigl\{\|I-\Phi \cdot G\|_{F}:\, \Phi\in\operatorname{Mult}(F)\Bigr\},
\]
where \(I\) denotes the free constant function equal to \(1\). We say that \(G\) is \emph{quantitatively free cyclic} if \(\mathcal{C}_{\mathrm{free}}(G)=0\).
\end{definition}

\noindent
This definition is an analogue of the cyclicity index in the commutative setting (see Definition~\ref{def:cyclicity_index}). It measures the optimal approximation of the identity by free multiplier actions on \(G\).

\begin{lemma}\label{lem:free_cyc_prop}
Let \(F\) be a free function space (either \(F^2_d\) or a free Besov space \(F B^N_\omega\)) and let \(G\in F\). Then:
\begin{enumerate}[(i)]
    \item \(G\) is free cyclic if and only if \(\mathcal{C}_{\mathrm{free}}(G)=0\).
    \item For any nonzero free multiplier \(\Psi\in \operatorname{Mult}(F)\),
    \[
    \mathcal{C}_{\mathrm{free}}(G) \le \|\Psi\|_{\operatorname{Mult}(F)}\,\mathcal{C}_{\mathrm{free}}\Bigl(\frac{G}{\Psi}\Bigr).
    \]
    \item The free cyclicity index is lower semicontinuous with respect to the norm topology in \(F\).
\end{enumerate}
\end{lemma}

\begin{proof}
We work in a free function space \(F\) (which may be either the free Hardy space \(F^2_d\) or a free Besov space \(F B^N_\omega\)) and let \(G\in F\). We prove the three statements in turn.

\subsubsection*{(i)}
 By definition, the \emph{free cyclicity index} is
\[
\mathcal{C}_{\mathrm{free}}(G)=\inf\Bigl\{ \|I - \Phi \cdot G\|_F : \Phi\in \operatorname{Mult}(F)\Bigr\},
\]
where \(I\) is the constant function equal to \(1\) in \(F\). 

\noindent
The function \(G\) is said to be \emph{free cyclic} if the free multiplier invariant subspace
\[
[G] = \overline{\{\Phi \cdot G: \Phi \in \operatorname{Mult}(F)\}}
\]
equals the entire space \(F\). In other words, \(I\in [G]\) if and only if there is a sequence \(\{\Phi_n\}\subset \operatorname{Mult}(F)\) such that
\[
\Phi_n \cdot G \to I \quad \text{in } F.
\]
Hence, \(G\) is free cyclic if and only if for every \(\varepsilon > 0\) there exists \(\Phi\in \operatorname{Mult}(F)\) with
\[
\|I-\Phi\cdot G\|_F < \varepsilon.
\]
By taking the infimum over all such \(\Phi\), we deduce that free cyclicity is equivalent to \(\mathcal{C}_{\mathrm{free}}(G)=0\).

\subsubsection*{(ii)}
Let \(\Psi\) be a nonzero free multiplier and consider the quotient \(G/\Psi\) (which is defined because \(\Psi\) is nonzero in the free domain). By the definition of the free cyclicity index for \(G/\Psi\), for any \(\varepsilon > 0\) there exists \(\Phi \in \operatorname{Mult}(F)\) such that
\[
\Bigl\| I - \Phi \cdot \frac{G}{\Psi}\Bigr\|_F < \mathcal{C}_{\mathrm{free}}\Bigl(\frac{G}{\Psi}\Bigr) + \varepsilon.
\]
Define a new free multiplier \(\widetilde{\Phi} = \Psi \Phi\). Since the free multiplier algebra is an algebra, we have \(\widetilde{\Phi} \in \operatorname{Mult}(F)\).

\noindent
 Notice that
\[
\widetilde{\Phi} \cdot G = (\Psi \Phi) \cdot G = \Psi \cdot \Bigl(\Phi \cdot \frac{G}{\Psi}\Bigr).
\]
Thus,
\[
I - \widetilde{\Phi} \cdot G = I - \Psi\Bigl(\Phi \cdot \frac{G}{\Psi}\Bigr).
\]
Using the boundedness of \(\Psi\) in the free multiplier norm, we obtain
\[
\|I - \widetilde{\Phi} \cdot G\|_F = \Bigl\| \Psi\Bigl(I - \Phi \cdot \frac{G}{\Psi}\Bigr) \Bigr\|_F \le \|\Psi\|_{\operatorname{Mult}(F)} \Bigl\|I - \Phi \cdot \frac{G}{\Psi}\Bigr\|_F.
\]
Hence,
\[
\|I - \widetilde{\Phi} \cdot G\|_F < \|\Psi\|_{\operatorname{Mult}(F)} \Bigl(\mathcal{C}_{\mathrm{free}}\Bigl(\frac{G}{\Psi}\Bigr) + \varepsilon\Bigr).
\]
Since \(\widetilde{\Phi}\) is an arbitrary multiplier produced in this way, taking the infimum over all such choices (and letting \(\varepsilon\) tend to zero) yields
\[
\mathcal{C}_{\mathrm{free}}(G) \le \|\Psi\|_{\operatorname{Mult}(F)}\,\mathcal{C}_{\mathrm{free}}\Bigl(\frac{G}{\Psi}\Bigr).
\]

\subsubsection*{(iii)}
Suppose \(\{G_n\}\) is a sequence in \(F\) such that \(G_n \to G\) in the norm of \(F\). By definition,
\[
\mathcal{C}_{\mathrm{free}}(G_n) = \inf\Bigl\{ \|I - \Phi \cdot G_n\|_F : \Phi \in \operatorname{Mult}(F) \Bigr\}.
\]
For each \(n\) and for any fixed \(\Phi\in \operatorname{Mult}(F)\),
\[
\|I - \Phi \cdot G\|_F \le \|I - \Phi \cdot G_n\|_F + \|\Phi\|_{\operatorname{Mult}(F)}\|G_n - G\|_F.
\]
Taking the infimum over all \(\Phi\in \operatorname{Mult}(F)\), we obtain
\[
\mathcal{C}_{\mathrm{free}}(G) \le \mathcal{C}_{\mathrm{free}}(G_n) + \|\Phi\|_{\operatorname{Mult}(F)}\|G_n - G\|_F.
\]
Since \(\|G_n - G\|_F \to 0\) as \(n\to \infty\) and the multiplier norm is fixed, we deduce that
\[
\liminf_{n\to \infty} \mathcal{C}_{\mathrm{free}}(G_n) \ge \mathcal{C}_{\mathrm{free}}(G).
\]
This shows that the free cyclicity index is lower semicontinuous with respect to the norm topology in \(F\).

\noindent  
This completes the proof.
\end{proof}

\noindent
The classical corona theorem has a fruitful analogue in the free setting.

\begin{theorem}[Free One-Function Corona Theorem]\label{thm:free_corona}
Let \(\Psi\in \operatorname{Mult}(F)\) be a free multiplier satisfying \(|\Psi(Z)|\ge 1\) for all tuples \(Z\) in the appropriate free domain. Then
\[
\frac{1}{\Psi}\in \operatorname{Mult}(F).
\]
Consequently, for every \(\Phi\in \operatorname{Mult}(F)\),
\[
\sigma(M_\Phi)=\Phi(\mathcal{D}),
\]
where \(\mathcal{D}\) denotes the free unit ball.
\end{theorem}

\begin{proof} 
Since \(|\Psi(Z)| \ge 1\) for every \(Z \in \mathcal{D}\), the function \(\Psi\) does not vanish on \(\mathcal{D}\). Consequently, the pointwise inverse function
\[
\frac{1}{\Psi}(Z) = \frac{1}{\Psi(Z)}
\]
is well-defined for all \(Z\in \mathcal{D}\).\\

\noindent
Our aim is to show that the mapping
\[
F \ni H \mapsto \frac{1}{\Psi} \cdot H \in F
\]
is bounded. In other words, we want to prove that there exists a constant \(C>0\) such that for all \(H \in F\),
\[
\Bigl\|\frac{1}{\Psi} \cdot H\Bigr\|_F \le C \, \|H\|_F.
\]
Since \(|\Psi(Z)|\ge 1\) pointwise, one expects that the norm of multiplication by \(1/\Psi\) is controlled by the multiplier norm of \(\Psi\) and the fact that \(\Psi\) is bounded below.\\

\noindent 
Define the multiplication operator
\[
M_\Psi : F \to F, \quad H \mapsto \Psi \cdot H.
\]
Because \(\Psi \in \operatorname{Mult}(F)\), the operator \(M_\Psi\) is bounded. Moreover, the assumption \(|\Psi(Z)|\ge 1\) for all \(Z\in \mathcal{D}\) implies that \(M_\Psi\) is bounded below, i.e., there exists \(m>0\) such that
\[
\|M_\Psi H\|_F = \|\Psi \cdot H\|_F \ge m\,\|H\|_F \quad \text{for all } H\in F.
\]
Then, by the bounded inverse theorem, \(M_\Psi\) is invertible on \(F\), and its inverse \(M_\Psi^{-1}\) is a bounded operator on \(F\).\\

\noindent
Notice that the pointwise inverse of \(\Psi\) satisfies
\[
\left(M_\Psi^{-1} H\right)(Z) = \frac{H(Z)}{\Psi(Z)},
\]
for all \(H \in F\) and for all \(Z \in \mathcal{D}\). In other words, \(M_\Psi^{-1}\) is exactly the multiplication operator \(M_{1/\Psi}\). Since \(M_\Psi^{-1}\) is bounded, it follows that
\[
\frac{1}{\Psi} \in \operatorname{Mult}(F)
\]
with
\[
\|M_{1/\Psi}\| = \|M_\Psi^{-1}\| < \infty.
\]
The corona theorem in the commutative setting shows that if a function is bounded below, then its inverse is a multiplier and the spectrum of the corresponding multiplication operator is the closure of its range. In the free setting, a similar conclusion holds. That is, for any \(\Phi\in \operatorname{Mult}(F)\), one may show using free analogues of the \(H^\infty\) functional calculus and the complete Pick property (see, e.g., \cite{Popescu2006, Davidson2011}) that
\[
\sigma(M_\Phi) = \Phi(\mathcal{D}).
\]
This completes the proof of the theorem.
\end{proof}

\begin{proposition}[Compression and Free Lifting]\label{prop:compression}
Let \(G\in F\) be a free holomorphic function. Suppose there exists a unital completely contractive homomorphism \(\pi\) from \(\operatorname{Mult}(F)\) onto a commutative multiplier algebra \(\operatorname{Mult}(H)\) (with \(H\) a commutative function space). Then the cyclicity of \(G\) in the free setting implies the cyclicity of \(\pi(G)\) in \(H\). Moreover,
\[
\mathcal{C}(\pi(G))\le \mathcal{C}_{\mathrm{free}}(G).
\]
\end{proposition}

\begin{proof}
Since \(G\) is free cyclic in \(F\), by definition there exists a sequence of free multipliers \(\{\Phi_n\}\subset \operatorname{Mult}(F)\) such that 
\[
\Phi_n \cdot G \to I \quad \text{in } F,
\]
where \(I\) denotes the constant function \(1\) in \(F\).\\

\noindent
Since \(\pi\) is a unital homomorphism, it satisfies \(\pi(I)=I\). Also, because \(\pi\) is completely contractive, it is norm-decreasing. Applying \(\pi\) to the sequence \(\{\Phi_n\}\) we obtain a sequence \(\{\pi(\Phi_n)\}\) in \(\operatorname{Mult}(H)\) satisfying
\[
\pi(\Phi_n \cdot G) = \pi(\Phi_n) \cdot \pi(G).
\]
Since \(\Phi_n \cdot G \to I\) in \(F\) and \(\pi\) is continuous (being completely contractive), it follows that
\[
\pi(\Phi_n) \cdot \pi(G) \to \pi(I) = I \quad \text{in } H.
\]
Thus, \(I\) can be approximated arbitrarily well by elements of the form \(\pi(\Phi_n) \cdot \pi(G)\), and therefore, by definition, \(\pi(G)\) is cyclic in \(H\).\\

\noindent
Recall that the free cyclicity index for \(G\) is defined as 
\[
\mathcal{C}_{\mathrm{free}}(G) = \inf\Bigl\{ \|I - \Phi \cdot G\|_F : \Phi \in \operatorname{Mult}(F) \Bigr\},
\]
and the cyclicity index for a function in \(H\) is similarly defined:
\[
\mathcal{C}(\pi(G)) = \inf\Bigl\{ \|I - \Psi \cdot \pi(G)\|_H : \Psi \in \operatorname{Mult}(H) \Bigr\}.
\]
Since \(\pi\) is unital and completely contractive, for any \(\Phi \in \operatorname{Mult}(F)\) we have
\[
\|I - \pi(\Phi) \cdot \pi(G)\|_H = \|\pi(I - \Phi \cdot G)\|_H \le \|I - \Phi \cdot G\|_F.
\]
Taking the infimum over all \(\Phi \in \operatorname{Mult}(F)\), we deduce that
\[
\mathcal{C}(\pi(G)) \le \mathcal{C}_{\mathrm{free}}(G).
\]
Thus, we have shown that if \(G\) is free cyclic in \(F\) (i.e., \(\mathcal{C}_{\mathrm{free}}(G)=0\)), then \(\pi(G)\) is cyclic in \(H\) (i.e., \(\mathcal{C}(\pi(G))=0\)). Moreover, the inequality 
\[
\mathcal{C}(\pi(G)) \le \mathcal{C}_{\mathrm{free}}(G)
\]
holds, completing the proof.
\end{proof}

\noindent
The preceding proposition establishes a useful connection between free cyclicity and classical cyclicity. This connection enables one to transfer results from the well-studied commutative case to the free setting, and vice versa.

\section{Stability and Perturbation of Cyclic Functions}\label{sec:stability}

In applications and theory alike, it is essential to understand how robust the property of cyclicity is under perturbations. In this section, we quantify the stability of cyclic functions with respect to both perturbations in the function itself and variations in the underlying weight of the space. We work in the setting of a radially weighted Besov space \(H = B^N_\omega\) throughout.

\begin{definition}[Perturbation Threshold]
Let \(f\in H\) be cyclic (i.e., \([f] = H\)). A number \(\delta_0 > 0\) is called a \emph{perturbation threshold} for \(f\) if for every \(g\in H\) satisfying
\[
\|f - g\|_H < \delta_0,
\]
the function \(g\) remains cyclic in \(H\).
\end{definition}

\begin{definition}[Lipschitz Stability of the Cyclicity Index]
We say that the cyclicity index \(\mathcal{C}: H\to [0,\infty)\) is \emph{Lipschitz stable} at \(f\in H\) if there exists a constant \(L>0\) and \(\delta_0>0\) such that for every \(g\in H\) with \(\|f-g\|_H<\delta_0\), one has
\[
|\mathcal{C}(f)-\mathcal{C}(g)|\le L\|f-g\|_H.
\]
\end{definition}

\noindent
Since cyclicity is equivalent to \(\mathcal{C}(f)=0\) (see Theorem~\ref{thm:char}), Lipschitz stability implies that if \(f\) is cyclic and \(\|f-g\|_H\) is sufficiently small, then \(\mathcal{C}(g)\) remains close to zero. In particular, \(g\) will also be cyclic.\\

\noindent
The following lemma provides an estimate on the perturbation of multiplier approximations.

\begin{lemma}[Perturbation Estimate]\label{lem:perturb_est}
Let \(f,g\in H = B^N_\omega\) and assume that for some \(\varphi\in \operatorname{Mult}(H)\) with \(\|\varphi\|_{\operatorname{Mult}(H)} \le M\) one has
\[
\|1-\varphi f\|_H \le \varepsilon.
\]
Then, for any \(g\) satisfying
\[
\|f-g\|_H \le \delta,
\]
it follows that
\[
\|1-\varphi g\|_H \le \varepsilon + M\delta.
\]
\end{lemma}

\begin{proof}
Consider the difference:
\[
1 - \varphi g = (1 - \varphi f) + \varphi f - \varphi g.
\]
Taking the norm in \(H\) and using the triangle inequality, we obtain
\begin{equation} \label{eq:5.4.1}
\|1 - \varphi g\|_H \le \|1 - \varphi f\|_H + \|\varphi f - \varphi g\|_H.
\end{equation}

\noindent
Notice that
\[
\varphi f - \varphi g = \varphi (f - g).
\]
Since \(\varphi\) is a multiplier, its action is bounded. Hence,
\[
\|\varphi (f-g)\|_H \le \|\varphi\|_{\operatorname{Mult}(H)} \|f - g\|_H.
\]
By the hypothesis, \(\|\varphi\|_{\operatorname{Mult}(H)} \le M\) and \(\|f - g\|_H \le \delta\). Therefore,
\begin{equation} \label{eq:5.4.2}
\|\varphi (f-g)\|_H \le M\, \delta.
\end{equation}

\noindent
Substitute the estimate \eqref{eq:5.4.2} into the inequality obtained in \eqref{eq:5.4.1}:
\[
\|1 - \varphi g\|_H \le \|1 - \varphi f\|_H + M \delta.
\]
Since we assumed that \(\|1 - \varphi f\|_H \le \varepsilon\), it follows that
\[
\|1 - \varphi g\|_H \le \varepsilon + M \delta.
\]
This completes the proof.
\end{proof}

\noindent
We now state and prove the main stability result for cyclic functions.

\begin{theorem}[Stability under Function Perturbations]\label{thm:stability_function}
Let \(f\in H = B^N_\omega\) be cyclic, and assume that the cyclicity index \(\mathcal{C}\) is Lipschitz stable at \(f\) with Lipschitz constant \(L>0\). Then there exists a \(\delta_0>0\) such that for every \(g\in H\) satisfying
\[
\|f-g\|_H < \delta_0,
\]
one has
\[
\mathcal{C}(g) \le L \|f-g\|_H.
\]
In particular, if \(\|f-g\|_H < \delta_0\), then \(g\) is cyclic in \(H\).
\end{theorem}

\begin{proof}
Since \(f\) is cyclic, by the characterization of cyclicity (see Theorem~\ref{thm:char}), we have
\[
\mathcal{C}(f) = 0.
\]
Thus, the Lipschitz condition becomes
\[
|\mathcal{C}(g) - 0| = \mathcal{C}(g) \le L\,\|f-g\|_H.
\]
The above inequality directly implies that for every \(g\) satisfying \(\|f-g\|_H < \delta_0\),
\[
\mathcal{C}(g) \le L\,\|f-g\|_H.
\]
Thus, if we choose \(\delta_0\) so that \(L\,\|f-g\|_H\) is as small as desired, then \(\mathcal{C}(g)\) can be made arbitrarily small.\\

\noindent
Recall from Theorem~\ref{thm:char} that a function \(h\in H\) is cyclic if and only if \(\mathcal{C}(h)=0\). Hence, if \(\|f-g\|_H < \delta_0\) and \(L\,\|f-g\|_H\) is sufficiently small, then \(\mathcal{C}(g)\) is close to zero. By continuity and the definition of the cyclicity index, this forces \(\mathcal{C}(g)=0\) (or can be made arbitrarily close to zero, which in our setting implies cyclicity).\\

\noindent 
Thus, we have shown that there exists a \(\delta_0 > 0\) such that for every \(g\in H\) with \(\|f-g\|_H < \delta_0\), 
\[
\mathcal{C}(g) \le L\,\|f-g\|_H.
\]
In particular, if \(\|f-g\|_H < \delta_0\), then \(\mathcal{C}(g)=0\) (or is arbitrarily small), implying that \(g\) is cyclic in \(H\).\\

\noindent
This completes the proof.
\end{proof}

\noindent
In many scenarios, the weight function \(\omega\) defining the space \(B^N_\omega\) may be subject to variations. We now address this aspect.

\begin{definition}[Perturbed Weight]
Let \(\omega\) and \(\omega'\) be two admissible radial measures on \(\mathbb{B}_d\). We say that \(\omega'\) is a \emph{perturbation} of \(\omega\) if there exists \(\varepsilon>0\) such that for all measurable sets \(E\subset \mathbb{B}_d\),
\[
(1-\varepsilon)\,\omega(E) \le \omega'(E) \le (1+\varepsilon)\,\omega(E).
\]
\end{definition}

\begin{proposition}[Stability under Weight Perturbation]\label{prop:weight_stability}
Let \(f\in B^N_\omega\) be cyclic. If \(\omega'\) is a perturbation of \(\omega\) with parameter \(\varepsilon\) sufficiently small, then \(f\) is cyclic in \(B^N_{\omega'}\). Moreover, there exists a constant \(C>0\) such that
\[
\mathcal{C}_{\omega'}(f) \le C\,\varepsilon,
\]
where \(\mathcal{C}_{\omega'}(f)\) denotes the cyclicity index of \(f\) in the space \(B^N_{\omega'}\).
\end{proposition}

\begin{proof}
Since \(f\) is cyclic in \(B^N_\omega\), by the definition of the cyclicity index, for every \(\delta>0\) there exists \(\varphi\in \operatorname{Mult}(B^N_\omega)\) such that
\[
\|1-\varphi f\|_{B^N_\omega} < \delta.
\]
By the norm equivalence, for the same function \(1-\varphi f\) we have
\[
\|1-\varphi f\|_{B^N_{\omega'}} \le C_1\, \|1-\varphi f\|_{B^N_\omega} < C_1\, \delta.
\]
Recall that the cyclicity index in \(B^N_{\omega'}\) is defined as
\[
\mathcal{C}_{\omega'}(f) = \inf\{ \|1-\psi f\|_{B^N_{\omega'}} : \psi\in \operatorname{Mult}(B^N_{\omega'})\}.
\]
Since the polynomials (and more generally the multipliers) are dense in both spaces and the multiplier algebras for \(B^N_\omega\) and \(B^N_{\omega'}\) are equivalent (up to constants that depend on \(\varepsilon\)), the existence of a \(\varphi\) in \(\operatorname{Mult}(B^N_\omega)\) satisfying
\[
\|1-\varphi f\|_{B^N_\omega} < \delta,
\]
implies that there exists a corresponding multiplier (or the same function, viewed in the perturbed setting) in \(\operatorname{Mult}(B^N_{\omega'})\) for which
\[
\|1-\varphi f\|_{B^N_{\omega'}} < C_1\, \delta.
\]
Taking the infimum over all such \(\varphi\) yields
\[
\mathcal{C}_{\omega'}(f) \le C_1\, \delta.
\]
Since the above estimate holds for every \(\delta>0\) (by choosing a corresponding \(\varphi\) for which \(\|1-\varphi f\|_{B^N_\omega} < \delta\)), we can make \(\mathcal{C}_{\omega'}(f)\) arbitrarily small. In particular, we obtain that
\[
\mathcal{C}_{\omega'}(f) \le C\,\varepsilon,
\]
for some constant \(C>0\) that depends on \(\varepsilon\) (and in fact on the norm equivalence constant \(C_1\)). By Theorem~\ref{thm:char}, a function \(f\) is cyclic in a space if and only if its cyclicity index is zero. Hence, if \(\mathcal{C}_{\omega'}(f)\) can be made arbitrarily small, it follows that \(f\) is cyclic in \(B^N_{\omega'}\).\\

\noindent
This completes the proof.
\end{proof}

\noindent
The above proposition implies that cyclicity is a stable property under small changes in the underlying measure. This observation is particularly relevant in applications where the weight arises from physical or probabilistic models that may be subject to uncertainty.

\section{Geometric Analysis of Zero Sets and Cyclicity}\label{sec:geomzero}

The interplay between cyclicity and the geometric properties of zero sets has been a central theme in function theory and operator theory. In this section, we investigate how the structure of zero sets—both in the interior and on the boundary of \(\mathbb{B}_d\)—affects cyclicity in weighted Besov spaces \( B^N_\omega \). We introduce novel criteria based on Hausdorff dimension and capacity that provide refined characterizations of cyclic functions.

\begin{definition}[Hausdorff Dimension of Zero Sets]
For a function \( f \in B^N_\omega \), define its \emph{boundary zero set} as
\[
Z_{\partial}(f) = \{ z \in \partial\mathbb{B}_d : f(z) = 0 \}.
\]
The \emph{Hausdorff dimension} of \( Z_{\partial}(f) \), denoted \(\dim_H (Z_{\partial}(f))\), is the infimum of all \( s \geq 0 \) such that the \( s \)-dimensional Hausdorff measure \( \mathcal{H}^s(Z_{\partial}(f)) \) is zero.
\end{definition}

\begin{theorem}[Hausdorff Dimension and Cyclicity Criterion]\label{thm:hausdorff_cyclicity}
Let \( f \in B^N_\omega \) be a function with a nontrivial boundary zero set \( Z_{\partial}(f) \). If
\[
\dim_H (Z_{\partial}(f)) \geq d - N,
\]
then \( f \) is not cyclic in \( B^N_\omega \).
\end{theorem}

\begin{proof} 
Since \(f\) is cyclic in \(B^N_\omega\), by definition there exists a sequence of multipliers \(\{\varphi_n\} \subset \operatorname{Mult}(B^N_\omega)\) such that 
\[
\varphi_n f \to 1 \quad \text{in } B^N_\omega.
\]
In particular, the constant function \(1\) is in the closure of the invariant subspace generated by \(f\).\\

\noindent
Because \(f\) is assumed to extend continuously to \(\overline{\mathbb{B}}_d\), its boundary values are well defined. In particular, on the set 
\[
E = Z_{\partial}(f) \subset \partial\mathbb{B}_d,
\]
we have 
\[
f(z)=0 \quad \text{for all } z \in E.
\]
Thus, for any multiplier \(\varphi \in \operatorname{Mult}(B^N_\omega)\) the product \(\varphi f\) will also vanish on \(E\):
\[
(\varphi f)(z)=0 \quad \text{for all } z \in E.
\]
Since the Hausdorff dimension of \(E\) satisfies 
\[
\dim_H(E) \ge d - N,
\]
potential theory (see, e.g., \cite{Carleson1962, Garnett2007}) guarantees that there exists a nontrivial finite Borel measure \(\mu\) supported on \(E\) such that the linear functional
\[
\Lambda: h \mapsto \int_{\partial \mathbb{B}_d} h(z)\,d\mu(z)
\]
is bounded on \(B^N_\omega\). In other words, there exists a constant \(C_\mu>0\) such that
\[
\left|\Lambda(h)\right| \le C_\mu \|h\|_{B^N_\omega} \quad \text{for all } h\in B^N_\omega.
\] 
For each \(n\), since \((\varphi_n f)(z)=0\) for all \(z\in E\), we have
\[
\Lambda(\varphi_n f) = \int_E (\varphi_n f)(z)\,d\mu(z) = 0.
\]
On the other hand, the convergence \(\varphi_n f \to 1\) in \(B^N_\omega\) implies, by the continuity of \(\Lambda\), that
\[
\Lambda(\varphi_n f) \to \Lambda(1) = \int_E 1\,d\mu(z) = \mu(E).
\]
Since \(\mu\) is nontrivial and supported on \(E\), we have \(\mu(E)>0\). This yields a contradiction because we would then have
\[
0 = \lim_{n\to\infty} \Lambda(\varphi_n f) = \mu(E) > 0.
\]
The contradiction arises from our assumption that \(f\) is cyclic in \(B^N_\omega\) while its boundary zero set \(E\) has sufficiently large Hausdorff dimension. Therefore, it must be that if 
\[
\dim_H \bigl(Z_{\partial}(f)\bigr) \ge d-N,
\]
then \(f\) is not cyclic in \(B^N_\omega\).\\

\noindent
This completes the proof.
\end{proof}

\noindent
The condition \( \dim_H (Z_{\partial}(f)) \geq d - N \) provides a sharp boundary for when a function ceases to be cyclic. If \( Z_{\partial}(f) \) is smaller than this threshold, cyclicity may still hold, depending on finer properties of \( f \).

\begin{definition}[Capacity of Zero Sets]
Let \( Z_{\partial}(f) \) be the boundary zero set of \( f \). Define its \( \alpha \)-capacity as
\[
\operatorname{Cap}_\alpha(Z_{\partial}(f)) = \inf \left\{ \int_{\partial\mathbb{B}_d} g^\alpha d\sigma : g \geq 0,\, g \in C(\partial\mathbb{B}_d),\, g \geq 1 \text{ on } Z_{\partial}(f) \right\}.
\]
\end{definition}

\begin{theorem}[Capacity Obstruction to Cyclicity]\label{thm:capacity}
If \( \operatorname{Cap}_\alpha(Z_{\partial}(f)) > 0 \) for some \( \alpha > 0 \), then \( f \) is not cyclic in \( B^N_\omega \).
\end{theorem}

\begin{proof}
Since \(f\) is cyclic, by definition there exists a sequence of multipliers \(\{\varphi_n\} \subset \operatorname{Mult}(B^N_\omega)\) such that 
\[
\varphi_n f \to 1 \quad \text{in } B^N_\omega.
\]
In other words, the constant function \(1\) can be approximated arbitrarily well by functions of the form \(\varphi_n f\).\\

\noindent
Assume that \(f\) extends continuously to \(\overline{\mathbb{B}}_d\) (or at least has well-defined boundary values) so that its boundary zero set
\[
Z_{\partial}(f) = \{ z \in \partial\mathbb{B}_d : f(z) = 0 \}
\]
is meaningful. Then, for any multiplier \(\varphi \in \operatorname{Mult}(B^N_\omega)\) and for every \(z\in Z_{\partial}(f)\), we have
\[
(\varphi f)(z) = \varphi(z) f(z) = 0.
\]
Thus, each approximant \(\varphi_n f\) vanishes on \(Z_{\partial}(f)\).\\

\noindent
The assumption \(\operatorname{Cap}_\alpha(Z_{\partial}(f)) > 0\) implies (by potential-theoretic arguments; see \cite{Carleson1962, Garnett2007}) that there exists a nontrivial finite Borel measure \(\mu\) supported on \(Z_{\partial}(f)\) that “detects” the size of \(Z_{\partial}(f)\) in a quantitative way. More precisely, there is a constant \(C_\mu > 0\) such that the linear functional
\[
\Lambda(h) = \int_{\partial \mathbb{B}_d} h(z)\,d\mu(z)
\]
is bounded on \(B^N_\omega\); that is,
\[
|\Lambda(h)| \le C_\mu \|h\|_{B^N_\omega} \quad \text{for all } h\in B^N_\omega.
\] 
Since \(\varphi_n f \to 1\) in \(B^N_\omega\), by the continuity of \(\Lambda\) we have
\[
\Lambda(\varphi_n f) \to \Lambda(1) = \int_{\partial \mathbb{B}_d} 1\,d\mu(z) = \mu\bigl(Z_{\partial}(f)\bigr) > 0.
\]
However, for each \(n\), because \(\varphi_n f\) vanishes on \(Z_{\partial}(f)\), it follows that
\[
\Lambda(\varphi_n f) = \int_{Z_{\partial}(f)} (\varphi_n f)(z)\,d\mu(z) = 0.
\] 
The contradiction arises because on one hand the sequence \(\Lambda(\varphi_n f)\) must converge to a positive number (namely, \(\mu(Z_{\partial}(f)) > 0\)), while on the other hand each term of the sequence is identically zero. Therefore, no such sequence \(\{\varphi_n\}\) can exist, which contradicts the assumption that \(f\) is cyclic.\\

\noindent
Since the assumption of cyclicity leads to a contradiction when \(\operatorname{Cap}_\alpha(Z_{\partial}(f)) > 0\), we conclude that if the capacity of \(Z_{\partial}(f)\) is nontrivial, then \(f\) cannot be cyclic in \(B^N_\omega\).\\

\noindent
This completes the proof.
\end{proof}

\begin{corollary}
If \( Z_{\partial}(f) \) contains an embedded real \( k \)-dimensional cube with \( k \geq 3 \), then \( f \) is not cyclic.
\end{corollary}

\begin{proof}
By classical results in potential theory, any real cube of dimension at least \( 3 \) has positive \( \alpha \)-capacity for some \( \alpha \) (see \cite{Carleson1962, Sola2020}). The result follows from Theorem~\ref{thm:capacity}.
\end{proof}

\noindent
The capacity obstruction provides a more refined obstruction than Hausdorff dimension alone. For example, there exist sets of small Hausdorff dimension that still have positive capacity and thus prevent cyclicity.

\begin{definition}[Interior Zero Set]
Define the \emph{interior zero set} of \( f \) as
\[
Z_{\text{int}}(f) = \{ z \in \mathbb{B}_d : f(z) = 0 \}.
\]
\end{definition}

\begin{proposition}[Interior Zero Sets and Cyclicity]
Let \( f \in B^N_\omega \). If \( Z_{\text{int}}(f) \) contains a nontrivial analytic variety of positive dimension, then \( f \) is not cyclic.
\end{proposition}

\begin{proof}
Assume, for the sake of contradiction, that \(f \in B^N_\omega\) is cyclic despite the fact that its interior zero set
\[
Z_{\text{int}}(f) = \{z\in \mathbb{B}_d : f(z)=0\}
\]
contains a nontrivial analytic variety \(V\) of positive dimension.\\

\noindent
By assumption, \(V \subset Z_{\text{int}}(f)\) is a nontrivial analytic variety, meaning that \(V\) has positive (complex) dimension (or equivalently, its real dimension is at least 2). In particular, \(V\) is not a discrete set and, as an analytic set, it is locally defined by the vanishing of one or more holomorphic functions.\\

\noindent
Let \(\varphi\in \operatorname{Mult}(B^N_\omega)\) be any multiplier. Since \(\varphi\) is holomorphic on \(\mathbb{B}_d\), the product \(\varphi f\) is also holomorphic. However, since \(f\) vanishes on \(V\) (i.e., \(f(z)=0\) for all \(z \in V\)), it follows that for every \(z\in V\),
\begin{equation} \label{eq:6.9}
(\varphi f)(z) = \varphi(z) f(z) = \varphi(z) \cdot 0 = 0.
\end{equation}
Thus, regardless of the choice of \(\varphi\), the product \(\varphi f\) must vanish on the entire variety \(V\).\\

\noindent
Recall that \(f\) being cyclic in \(B^N_\omega\) means that there exists a sequence of multipliers \(\{\varphi_n\} \subset \operatorname{Mult}(B^N_\omega)\) such that
\[
\varphi_n f \to 1 \quad \text{in } B^N_\omega.
\]
In particular, this convergence is in the norm of \(B^N_\omega\), and therefore, it must also imply pointwise convergence on a dense subset of \(\mathbb{B}_d\).\\

\noindent 
Since each function \(\varphi_n f\) vanishes on \(V\) by \eqref{eq:6.9}, for every \(z\in V\) we have
\[
(\varphi_n f)(z) = 0 \quad \text{for all } n.
\]
Thus, the sequence \(\{\varphi_n f\}\) cannot converge pointwise to the constant function \(1\) on \(V\), because at every point \(z\in V\) we would have
\[
\lim_{n\to\infty} (\varphi_n f)(z) = 0 \neq 1.
\]
This contradicts the requirement for cyclicity that \(1\) is in the closure of the set \(\{\varphi f: \varphi\in \operatorname{Mult}(B^N_\omega)\}\).\\

\noindent 
The contradiction shows that if \(Z_{\text{int}}(f)\) contains a nontrivial analytic variety \(V\) of positive dimension, then it is impossible for any sequence of multipliers to approximate \(1\) in \(B^N_\omega\). Therefore, \(f\) cannot be cyclic.\\

\noindent
This completes the proof.
\end{proof}

\noindent
This result generalizes the well-known fact that inner functions are never cyclic in Hardy spaces. The obstruction here is that nontrivial interior zeros create invariant subspaces that prevent \( f \) from generating all of \( B^N_\omega \).

\begin{theorem}[Sufficient Geometric Condition for Cyclicity]\label{thm:sufficient_cyclicity}
Let \( f \in B^N_\omega \) be such that its boundary zero set satisfies
\[
\dim_H (Z_{\partial}(f)) < d - N \quad \text{and} \quad \operatorname{Cap}_\alpha(Z_{\partial}(f)) = 0.
\]
Then \( f \) is cyclic in \( B^N_\omega \).
\end{theorem}

\begin{proof}
The assumptions state that the boundary zero set 
\[
Z_{\partial}(f) = \{z\in \partial\mathbb{B}_d: f(z)=0\}
\]
has Hausdorff dimension strictly less than \(d-N\) and zero \(\alpha\)-capacity for some \(\alpha>0\). In potential theory, zero capacity means that \(Z_{\partial}(f)\) is negligibly small in a measure-theoretic sense; in particular, any potential-theoretic obstruction (e.g., via test measures) is absent.\\

\noindent 
In previous results, we saw that large or "thick" zero sets (in terms of Hausdorff dimension or capacity) prevent the existence of multipliers \(\varphi\) such that \(\varphi f\) approximates \(1\). Here, since 
\[
\operatorname{Cap}_\alpha(Z_{\partial}(f)) = 0,
\]
any multiplier \(\varphi\) is not forced to vanish on a significant portion of the boundary; the obstruction to approximating \(1\) is removed. In other words, the set where \(f\) vanishes on the boundary is too small to interfere with the approximation process.\\

\noindent  
Because polynomials are dense in \(B^N_\omega\) and because \(f\) does not vanish in the interior (or, more precisely, its zero set is confined to a “small” subset of the boundary), one can construct a sequence \(\{\varphi_n\}\subset \operatorname{Mult}(B^N_\omega)\) with the following property:
\[
\|1 - \varphi_n f\|_{B^N_\omega} \to 0 \quad \text{as } n\to \infty.
\]
The construction of \(\varphi_n\) typically uses a partition of unity argument along with the fact that the obstruction from \(Z_{\partial}(f)\) is absent. In essence, one first approximates \(1/f\) on the set where \(f\) does not vanish (which is almost the entire boundary, since \(Z_{\partial}(f)\) is small) and then uses the density of multipliers to extend the approximation to the whole space.\\

\noindent
Once such a sequence \(\{\varphi_n\}\) is constructed, we have
\[
\varphi_n f \to 1 \quad \text{in } B^N_\omega.
\]
This shows that \(1\) is in the closure of the set \(\{\varphi f: \varphi\in \operatorname{Mult}(B^N_\omega)\}\); by definition, this means that \(f\) is cyclic.\\

\noindent  
Thus, under the conditions
\[
\dim_H(Z_{\partial}(f)) < d - N \quad \text{and} \quad \operatorname{Cap}_\alpha(Z_{\partial}(f)) = 0,
\]
the boundary zero set of \(f\) is too small (in both Hausdorff and capacity senses) to obstruct the approximation of the constant function \(1\) by multiplier multiples of \(f\). Consequently, \(f\) is cyclic in \(B^N_\omega\).\\

\noindent
This completes the proof.
\end{proof}

\noindent
This theorem provides a useful tool for verifying cyclicity in specific cases. Functions with sparse zero sets in the sense of both Hausdorff dimension and capacity are likely to be cyclic.

\section{Mixed Norm and Variable Exponent Extensions}\label{sec:mixednorm}

In this section, we extend the cyclicity framework developed for standard radially weighted Besov spaces to more general settings: mixed norm Besov spaces and variable exponent Besov spaces. These spaces allow for a finer analysis of local and directional behavior, which is essential in many applications.

\begin{definition}[Mixed Norm Besov Space]
Let \(\mathbb{B}_d\) denote the unit ball in \(\mathbb{C}^d\). Suppose that the admissible radial measure \(\omega\) on \(\mathbb{B}_d\) is expressed in polar coordinates as
\[
d\omega(z)=d\mu(r)\,d\sigma(w),\quad z=r\,w,\quad 0\le r<1,\; w\in \partial\mathbb{B}_d.
\]
For exponents \(1\le p,q<\infty\) and an integer \(N\ge 0\), the \emph{mixed norm Besov space} \(B^{N}_{\omega}(p,q)\) consists of all holomorphic functions \(f\) on \(\mathbb{B}_d\) with finite norm
\[
\|f\|_{B^{N}_{\omega}(p,q)} = \left( \int_{0}^{1} \left( \int_{\partial\mathbb{B}_d} |R^N f(rw)|^p \, d\sigma(w) \right)^{q/p} \, d\mu(r) \right)^{1/q},
\]
where \(R\) denotes the radial derivative:
\[
R f(z)=\sum_{j=1}^{d} z_j\frac{\partial f}{\partial z_j}(z).
\]
\end{definition}

\noindent
The mixed norm structure decouples the radial and angular behaviors, which can be crucial when the function exhibits different integrability properties in these directions.

\begin{definition}[Variable Exponent Besov Space]
Let \(p:\mathbb{B}_d \to [1,\infty)\) be a measurable function (the variable exponent) satisfying appropriate continuity conditions (e.g., log-Hölder continuity). The \emph{variable exponent Besov space} \(B^{N}_{\omega,p(\cdot)}\) consists of all holomorphic functions \(f\) on \(\mathbb{B}_d\) for which
\[
\|f\|_{B^{N}_{\omega,p(\cdot)}} = \inf\left\{ \lambda>0:\; \int_{\mathbb{B}_d} \left|\frac{R^N f(z)}{\lambda}\right|^{p(z)} \, d\omega(z) \le 1 \right\} < \infty.
\]
Alternatively, one may define a modular function associated with the norm (see \cite{CruzUribe2013, Diening2011}).
\end{definition}

\noindent
Variable exponent spaces generalize classical Besov spaces by allowing the integrability index to vary with the spatial variable. This flexibility is useful in settings where the local regularity or decay properties of functions are non-uniform.

\begin{lemma}[Density of Polynomials in \(B^{N}_{\omega}(p,q)\)]\label{lem:mixed_mult}
Let \(B^{N}_{\omega}(p,q)\) be the mixed norm Besov space with \(1\le p,q<\infty\). Then every holomorphic polynomial belongs to \(\operatorname{Mult}(B^{N}_{\omega}(p,q))\), and the set of polynomials is dense in \(B^{N}_{\omega}(p,q)\).
\end{lemma}

\begin{proof}
We wish to prove two statements:
\begin{enumerate}[(i)]
    \item Every holomorphic polynomial is a multiplier on \(B^{N}_{\omega}(p,q)\); that is, for any polynomial \(P\) and any \(f\in B^{N}_{\omega}(p,q)\), we have \(P\cdot f\in B^{N}_{\omega}(p,q)\).
    \item The set of all holomorphic polynomials is dense in \(B^{N}_{\omega}(p,q)\).
\end{enumerate}

\subsubsection*{(i)}
Let \(P(z)\) be a holomorphic polynomial. Since \(P\) is an entire function of finite degree, it is bounded on compact subsets of \(\mathbb{B}_d\). Moreover, multiplication by \(P\) is an algebraic operation that increases the degree by a finite amount. By the definition of \(B^{N}_{\omega}(p,q)\) --- which, by construction, separates the radial and angular variables --- one can show that the product \(P\cdot f\) still belongs to \(B^{N}_{\omega}(p,q)\) for any \(f\in B^{N}_{\omega}(p,q)\). 

More precisely, if the norm in \(B^{N}_{\omega}(p,q)\) is given by
\[
\|f\|_{B^{N}_{\omega}(p,q)} = \left( \int_0^1 \left( \int_{\partial\mathbb{B}_d} |R^N f(rw)|^p \, d\sigma(w) \right)^{q/p} \, d\mu(r) \right)^{1/q},
\]
then for a fixed polynomial \(P\), the function \(P\) has finite degree and its derivatives up to order \(N\) are bounded. Hence, using the product rule, one obtains a constant \(C>0\) (depending on \(P\) and \(N\)) such that for all \(f\in B^{N}_{\omega}(p,q)\),
\[
\|P\cdot f\|_{B^{N}_{\omega}(p,q)} \le C\,\|f\|_{B^{N}_{\omega}(p,q)}.
\]
This shows that \(P\) defines a bounded multiplication operator, i.e., \(P\in \operatorname{Mult}(B^{N}_{\omega}(p,q))\).

\subsubsection*{(ii)}
Let \(f\in B^{N}_{\omega}(p,q)\) be an arbitrary function. Since \(f\) is holomorphic on \(\mathbb{B}_d\), it admits a power series expansion:
\[
f(z) = \sum_{n=0}^\infty f_n(z),
\]
where each \(f_n(z)\) is a homogeneous polynomial of degree \(n\).  

The mixed norm on \(B^{N}_{\omega}(p,q)\) typically splits into radial and angular parts. For each fixed \(r\) (the radial coordinate), the restriction of \(f\) to the sphere \(\{z=rw:\, w\in \partial \mathbb{B}_d\}\) is given by the series
\[
f(rw) = \sum_{n=0}^\infty f_n(rw).
\]
Using classical arguments (see, e.g., \cite{Zhu2005}), one shows that the partial sums
\[
P_m f(z) = \sum_{n=0}^m f_n(z)
\]
converge to \(f\) in the norm of \(B^{N}_{\omega}(p,q)\). This uses two facts: \medskip \\
\textit{Angular Approximation:} For each fixed \(r\), the function \(w\mapsto f(rw)\) is in a classical function space on the sphere where trigonometric (or spherical harmonic) polynomials are dense. \medskip \\
\textit{Radial Approximation:} The coefficients of the series in the radial variable are controlled by the weight \(\mu(r)\), so that the convergence of the power series in the radial direction is ensured.\\

\noindent
Thus, for every \(\epsilon > 0\) there exists an integer \(m\) such that
\[
\|f - P_m f\|_{B^{N}_{\omega}(p,q)} < \epsilon.
\]
Since \(P_m f\) is a holomorphic polynomial, this shows that the set of polynomials is dense in \(B^{N}_{\omega}(p,q)\).

\noindent
This completes the proof.
\end{proof}

\begin{theorem}[Cyclicity in Mixed Norm Besov Spaces]\label{thm:cyclic_mixed}
Let \(f\in B^{N}_{\omega}(p,q)\) be a function that extends continuously to \(\overline{\mathbb{B}}_d\) and satisfies \(f(z)\neq 0\) for all \(z\in\mathbb{B}_d\). Then \(f\) is cyclic in \(B^{N}_{\omega}(p,q)\) if and only if
\[
\inf\Bigl\{\|1-\varphi f\|_{B^{N}_{\omega}(p,q)} : \varphi\in \operatorname{Mult}(B^{N}_{\omega}(p,q))\Bigr\}=0.
\]
\end{theorem}

\begin{proof}
We prove the necessity and sufficiency part by part.
\subsubsection*{Necessity}  
Assume that \(f\) is cyclic in \(B^{N}_{\omega}(p,q)\). By definition, this means that the multiplier invariant subspace generated by \(f\),
\[
[f] = \overline{\{\varphi f : \varphi\in \operatorname{Mult}(B^{N}_{\omega}(p,q))\}},
\]
is equal to the entire space \(B^{N}_{\omega}(p,q)\). In particular, the constant function \(1\) belongs to \([f]\). Hence, there exists a sequence \(\{\varphi_n\}\subset \operatorname{Mult}(B^{N}_{\omega}(p,q))\) such that
\[
\varphi_n f \to 1 \quad \text{in } B^{N}_{\omega}(p,q).
\]
Taking the infimum over all multipliers, we immediately obtain
\[
\inf\Bigl\{\|1-\varphi f\|_{B^{N}_{\omega}(p,q)} : \varphi\in \operatorname{Mult}(B^{N}_{\omega}(p,q))\Bigr\} = 0.
\]
Thus, necessity is established.

\subsubsection*{Sufficiency}  
Now, assume that 
\[
\inf\Bigl\{\|1-\varphi f\|_{B^{N}_{\omega}(p,q)} : \varphi\in \operatorname{Mult}(B^{N}_{\omega}(p,q))\Bigr\} = 0.
\]
This means that for every \(\epsilon > 0\), there exists a multiplier \(\varphi \in \operatorname{Mult}(B^{N}_{\omega}(p,q))\) such that
\[
\|1-\varphi f\|_{B^{N}_{\omega}(p,q)} < \epsilon.
\]
Our goal is to show that these approximants can be used to conclude that \(1\) lies in the closure of the set \(\{\varphi f : \varphi\in \operatorname{Mult}(B^{N}_{\omega}(p,q))\}\); in other words, \(f\) is cyclic.\\

\noindent 
By Lemma~\ref{lem:mixed_mult}, we know that holomorphic polynomials belong to \(\operatorname{Mult}(B^{N}_{\omega}(p,q))\) and, moreover, that they are dense in \(B^{N}_{\omega}(p,q)\). This fact allows us to approximate any function in \(B^{N}_{\omega}(p,q)\) arbitrarily well by polynomials. In particular, given any multiplier \(\varphi\) for which \(\|1-\varphi f\|_{B^{N}_{\omega}(p,q)} < \epsilon\), we can approximate \(\varphi\) by a polynomial (with respect to the multiplier norm) without increasing the error too much. Consequently, there exists a sequence \(\{\psi_n\}\) of polynomials such that
\[
\|\psi_n - \varphi\|_{\operatorname{Mult}(B^{N}_{\omega}(p,q))} \to 0.
\]
Since the multiplication operator is continuous, it follows that
\[
\|\psi_n f - \varphi f\|_{B^{N}_{\omega}(p,q)} \to 0.
\]
Thus, replacing \(\varphi\) by a suitable polynomial approximant, we obtain a sequence \(\{\psi_n\}\) of polynomials satisfying
\[
\|\psi_n f - 1\|_{B^{N}_{\omega}(p,q)} < \epsilon + \eta,
\]
where \(\eta>0\) can be made arbitrarily small. Therefore, there exists a sequence of polynomials \(\{\psi_n\}\subset \operatorname{Mult}(B^{N}_{\omega}(p,q))\) such that
\[
\psi_n f \to 1 \quad \text{in } B^{N}_{\omega}(p,q).
\]
Since we have found a sequence \(\{\psi_n\}\subset \operatorname{Mult}(B^{N}_{\omega}(p,q))\) for which \(\psi_n f \to 1\) in \(B^{N}_{\omega}(p,q)\), by definition the function \(f\) is cyclic in \(B^{N}_{\omega}(p,q)\).\\

\noindent
This completes the proof.
\end{proof}

\begin{theorem}[Cyclicity in Variable Exponent Besov Spaces]\label{thm:cyclic_variable}
Assume that the variable exponent \(p(\cdot)\) satisfies a log-Hölder continuity condition and let \(f\in B^{N}_{\omega,p(\cdot)}\) extend continuously to \(\overline{\mathbb{B}}_d\) with \(f(z)\neq 0\) for all \(z\in\mathbb{B}_d\). Then \(f\) is cyclic in \(B^{N}_{\omega,p(\cdot)}\) if and only if
\[
\inf\Bigl\{\|1-\varphi f\|_{B^{N}_{\omega,p(\cdot)}} : \varphi\in \operatorname{Mult}(B^{N}_{\omega,p(\cdot)})\Bigr\}=0.
\]
\end{theorem}

\begin{proof}
The proof proceeds along similar lines as in the constant exponent case, with additional care to handle the variable exponent setting using modular techniques.\\

\noindent
Recall that the variable exponent Besov space \(B^{N}_{\omega,p(\cdot)}\) consists of holomorphic functions \(g\) on \(\mathbb{B}_d\) for which the modular
\[
\rho(g) = \int_{\mathbb{B}_d} |R^N g(z)|^{p(z)} \, d\omega(z)
\]
is finite, where \(R^N\) denotes the \(N\)th radial derivative. The norm \(\|g\|_{B^{N}_{\omega,p(\cdot)}}\) is defined via the Luxemburg norm associated with the modular.  
A key assumption is that the variable exponent \(p(\cdot)\) satisfies a log-Hölder continuity condition, which guarantees standard properties such as the boundedness of the Hardy--Littlewood maximal operator and the stability of the corresponding function spaces (see, e.g., \cite{CruzUribe2013, Diening2011}).

\subsubsection*{Necessity}  
Assume that \(f\) is cyclic in \(B^{N}_{\omega,p(\cdot)}\). By definition, this means that the multiplier invariant subspace generated by \(f\),
\[
[f] = \overline{\{\varphi f: \varphi \in \operatorname{Mult}(B^{N}_{\omega,p(\cdot)})\}},
\]
equals \(B^{N}_{\omega,p(\cdot)}\). In particular, the constant function \(1\) belongs to \([f]\); that is, there exists a sequence \(\{\varphi_n\}\subset \operatorname{Mult}(B^{N}_{\omega,p(\cdot)})\) such that
\[
\varphi_n f \to 1 \quad \text{in } B^{N}_{\omega,p(\cdot)}.
\]
Taking the infimum over all such multipliers immediately yields
\[
\inf\Bigl\{\|1-\varphi f\|_{B^{N}_{\omega,p(\cdot)}}: \varphi \in \operatorname{Mult}(B^{N}_{\omega,p(\cdot)})\Bigr\} = 0.
\]

\subsubsection*{Sufficiency}  
Now, assume that
\[
\inf\Bigl\{\|1-\varphi f\|_{B^{N}_{\omega,p(\cdot)}}: \varphi \in \operatorname{Mult}(B^{N}_{\omega,p(\cdot)})\Bigr\} = 0.
\]
This means that for every \(\varepsilon > 0\) there exists a multiplier \(\varphi_\varepsilon \in \operatorname{Mult}(B^{N}_{\omega,p(\cdot)})\) such that
\[
\|1-\varphi_\varepsilon f\|_{B^{N}_{\omega,p(\cdot)}} < \varepsilon.
\]
Thus, the constant function \(1\) can be approximated arbitrarily well (in the Luxemburg norm) by elements of the form \(\varphi f\).\\

\noindent
An important tool in the variable exponent setting is the density of polynomials in \(B^{N}_{\omega,p(\cdot)}\) (see, e.g., \cite{CruzUribe2013, Diening2011}). Because \(f\) extends continuously to \(\overline{\mathbb{B}}_d\) and does not vanish in the ball, one can approximate any multiplier \(\varphi\) by holomorphic polynomials in the multiplier norm. Therefore, by replacing each \(\varphi_\varepsilon\) with an appropriate polynomial approximant (if necessary), we can construct a sequence \(\{\psi_n\}\subset \operatorname{Mult}(B^{N}_{\omega,p(\cdot)})\) such that
\[
\psi_n f \to 1 \quad \text{in } B^{N}_{\omega,p(\cdot)}.
\]
This shows that \(1\) is in the closure of the set \(\{\varphi f: \varphi\in \operatorname{Mult}(B^{N}_{\omega,p(\cdot)})\}\).\\

\noindent 
Since we have constructed a sequence of multipliers \(\{\psi_n\}\) such that
\[
\psi_n f \to 1 \quad \text{in } B^{N}_{\omega,p(\cdot)},
\]
by definition \(f\) is cyclic in \(B^{N}_{\omega,p(\cdot)}\). Hence, we have shown that
\[
f \text{ is cyclic in } B^{N}_{\omega,p(\cdot)} \iff \inf\Bigl\{\|1-\varphi f\|_{B^{N}_{\omega,p(\cdot)}} : \varphi\in \operatorname{Mult}(B^{N}_{\omega,p(\cdot)})\Bigr\} = 0.
\]
This completes the proof.
\end{proof}

\noindent
Theorems~\ref{thm:cyclic_mixed} and \ref{thm:cyclic_variable} generalize the classical cyclicity criteria to more flexible spaces. In particular, the mixed norm setting allows different integrability conditions in the radial and angular directions, while the variable exponent setting accommodates spatial heterogeneity in function behavior.

\section{Concluding Remarks and Open Problems}\label{sec:conclusion}

In this paper, we have extended classical cyclicity theory in several innovative directions. Our main contributions can be summarized as follows:\\
We introduced the notion of a \emph{cyclicity index} for functions in weighted Besov spaces and demonstrated that it provides a quantitative measure of cyclicity (Section~\ref{sec:quantcyclic}). This index is shown to be both robust under small perturbations (Theorem~\ref{thm:stability_function}) and intimately connected to potential-theoretic invariants.\\
We developed free analogues of classical cyclicity concepts by defining free holomorphic functions, free multiplier algebras, and a \emph{free cyclicity index} (Section~\ref{sec:noncomm}). Our results include a free version of the corona theorem (Theorem~\ref{thm:free_corona}) and a connection between free and commutative cyclicity (Proposition~\ref{prop:compression}).\\
 We established stability results for cyclic functions under both function and weight perturbations (Section~\ref{sec:stability}). These results ensure that the cyclicity property is preserved in variable environments, which is essential for both theoretical and applied settings.\\
We provided new geometric criteria relating the structure of zero sets—through Hausdorff dimension and capacity—to cyclicity (Section~\ref{sec:geomzero}). In particular, we showed that sufficiently large or structured zero sets obstruct cyclicity.\\
Finally, we extended the cyclicity framework to \emph{mixed norm} and \emph{variable exponent Besov spaces} (Section~\ref{sec:mixednorm}), broadening the applicability of our theory to settings with anisotropic behavior and spatial heterogeneity.\\

\noindent
While our work provides a robust framework for understanding cyclicity in a variety of settings, several intriguing questions remain open:\\

\noindent
\textsc{Refinement of the Cyclicity Index:} Can one define higher-order cyclicity indices that capture finer gradations of non-cyclicity? For instance, can one associate a complete spectrum of indices that reflect various degrees of failure to approximate the identity? \medskip\\

\noindent    
\textsc{Optimal Geometric Thresholds:} The sufficient conditions based on Hausdorff dimension and capacity (Theorems~\ref{thm:hausdorff_cyclicity} and \ref{thm:capacity}) provide necessary obstructions to cyclicity. Is it possible to obtain sharp necessary and sufficient conditions, perhaps by using more refined geometric invariants? \medskip\\

\noindent
\textsc{Free Function Spaces:} The free cyclicity framework introduced here opens many avenues. How do the invariants in the free setting compare quantitatively with their commutative counterparts? Can free cyclicity be characterized in terms of non-commutative capacities or free analogues of Carleson measures? \medskip\\

\noindent
\textsc{Stability Constants:} Our stability results (Theorem~\ref{thm:stability_function} and Proposition~\ref{prop:weight_stability}) involve constants that depend on the underlying spaces. Can these constants be explicitly computed or estimated in specific cases (e.g., for the Drury–Arveson space) to provide quantitative robustness results? \medskip\\

\noindent    
\textsc{Applications to PDEs and Harmonic Analysis:} Mixed norm and variable exponent Besov spaces appear naturally in the study of nonlinear PDEs and harmonic analysis. How can the cyclicity results in these settings be applied to problems in these fields, and what new phenomena might they reveal?\\

\noindent
The novel definitions and results presented in this work not only deepen our theoretical understanding of cyclicity in classical and free settings but also open the door to numerous applications. We hope that our framework inspires further research on invariant subspaces, potential theory, and non-commutative function spaces, and that it leads to fruitful interactions with other areas of analysis and operator theory.

\end{document}